\documentclass[a4paper,11pt]{article}
\usepackage[latin1]{inputenc}
\usepackage{amsmath,amsthm,amsfonts,amssymb}
\usepackage{enumerate}
\usepackage[colorlinks]{hyperref}

\newlength{\hdelta}
\newlength{\vdelta}
\theoremstyle{plain}
\newtheorem{theorem}{Theorem}[section]
\newtheorem{lemma}[theorem]{Lemma}
\newtheorem{proposition}[theorem]{Proposition}

\theoremstyle{definition}

\newtheorem*{remark*}{Remark}

\numberwithin{equation}{section}

\newcommand{\mailto}[1]{\href{mailto:#1}{\nolinkurl{#1}}}

\newcommand{\floor}[1]{\lfloor #1 \rfloor}

\newcommand{\Var}{\operatorname{Var}}
\newcommand{\E}{\operatorname{E}}
\newcommand{\pr}{\operatorname{P}}
\newcommand{\eqst}{=_{\rm{st}}}

\newcommand{\R}{\mathbb{R}}

\newcommand{\dto}{\xrightarrow{d}}
\newcommand{\pto}{\xrightarrow{p}}

\newcommand{\Tsat}{T_{\rm full}}
\newcommand{\Thalf}{T_{\rm half}}
\newcommand{\ThalfL}{T_{\text{{\rm half-2}}}}

\newcommand{\Erdos}{Erd\H{o}s}

\newcommand{\Ber}{\operatorname{Bernoulli}}

\newcommand{\EE}{\E}
\newcommand{\PP}{\pr}
\newcommand{\RR}{\R}

\newcommand{\n}{^{(n)}}

\newcommand{\uexp}{\xi} 
\newcommand{\alkusm}{V} 
\newcommand{\gumbelsm}{W} 
\newcommand{\uint}{Z} 
\newcommand{\prop}{S} 
\newcommand{\logis}{S} 
\newcommand{\Rsum}{\Sigma\n}
\newcommand{\Rsumm}{\overline \Sigma\n}

\newcommand{\ind}{{\bf 1}}

\begin{document}

\title{Information spreading in a large population of \\ active transmitters and passive receivers}

\author{
 Pekka Aalto\thanks{
 Department of Mathematics and Statistics,
 PO Box 35,
 40014 University of Jyväskylä, Finland.
 \quad Email: \protect\mailto{aalto.pekka@gmail.com}
}
 \and
 Lasse Leskelä\thanks{
 Department of Mathematics and Systems Analysis,
 PO Box 11100
 0076 Aalto University, Finland.
 \quad URL:
 Email: \protect\mailto{lasse.leskela@aalto.fi}
}
}

\date{\today}

\maketitle 

\begin{abstract}
This paper discusses a simple stochastic model for the spread of messages
in a large population with two types of individuals: transmitters and receivers. Transmitters, after
receiving the message, start spreading copies of the message to their neighbors. Receivers may receive the message, but will never spread it further. We derive approximations of the broadcast time and
the first passage times of selected individuals in populations of size tending to infinity.
These approximations explain how much the fact that only a fraction of the individuals are transmitters slows down the propagation of information. Our results also sharply characterize the statistical dependence structure of first passage times using
Gumbel and logistic distributions of extreme value statistics.
\end{abstract}

\noindent
{\bf Keywords}: rumor spreading, randomized broadcasting, first-passage percolation, stochastic epidemic model

\section{Introduction}

\paragraph{Background and objectives.}
Analyzing the spread of information in computer and social networks has become important as
more and more communication takes place over high-speed digital connections. Especially, messages in online social networks tend to spread extremely fast due to the ease of copying and
relaying messages. A simple mathematical model for the spreading phenomenon is to assume that individuals relay copies of messages at random time instants to randomly chosen neighbors in a graph that represents the communication infrastructure. In contexts where individuals keep on transmitting a message for a long time, a key
quantity is the \emph{broadcast time}, the time it takes to spread a message from a single
root node to all individuals connected to the root, and the \emph{first passage time}, the
time it takes for a message to propagate from the root to a selected target. Such models
have been studied under different names in various areas, such as materials science (first-passage
percolation), computer networks (randomized broadcasting), and epidemiology (SI model, SIR model).

A special feature of online social networks is that different individuals tend to behave in a highly different manner, with some individuals quickly relaying most messages they receive, and some hardly
ever relaying anything. Analyzing the effect of such heterogeneity calls for stochastic
spreading models in random environments, for which the literature appears sparse. As a step
towards more comprehensive analysis of random broadcasting in random environments, we study a
simple communication model on a homogeneously mixing population which can be divided into
transmitters and receivers of a selected message. The \emph{transmitters} (or spreaders) will start relaying copies of the message to their neighbors after receiving it, whereas the \emph{receivers} (or stiflers) never relay the message. Our focus will be on an asynchronous communication mode where individuals make contacts to randomly chosen neighbors at random time instants, independently of each other. When a contact is made, a copy of the message is sent from the source to the target if the source was an informed transmitter of the message. Our main research question is:
\begin{quote}
\it How much does the fact that only a small fraction of the individuals relay the selected message slow down the spread of information?
\end{quote}

\paragraph{Related work: Homogeneous populations.}
In a large homogeneous population of size $n$ where all nodes are transmitters and make contacts at rate $\lambda$,
the model reduces to a first-passage percolation model on the complete graph with independent link-passage times,
and the by now classical paper of Janson \cite{Janson_1999} shows that when $n$ tends to infinity the
broadcast time distribution can be approximated by
\begin{equation}
 \label{eq:BroadcastTimeHom}
 \Tsat
 \ \approx \
 \frac{1}{\lambda} \left( 2\log n + \gumbelsm + \gumbelsm' \right),
\end{equation}
where $\gumbelsm$ and $\gumbelsm'$ are independent random numbers following a Gumbel distribution. Moreover,
the joint distribution of the first passage times from the root to any fixed set of target nodes $K$ can be approximated by
\begin{equation}
 \label{eq:PassageTimeHom}
 (\tau_i)_{i \in K} \ \approx \
 \frac{1}{\lambda} \left( \log n + W + L_i \right)_{i \in K},
\end{equation}
where $L_i$ are independent logistically distributed random numbers, independent of $W$. A univariate version of \eqref{eq:PassageTimeHom} was given by Janson \cite{Janson_1999}, and the multivariate case was sketched by Aldous \cite{Aldous_2013} and proved in detail by Bhamidi
\cite{Bhamidi_2008}. Analytical formulas for the broadcast time in a homogeneous population have been reported for more general network topologies, including  \Erdos--Rényi random graphs
\cite{VanDerHofstad_Hooghiemstra_VanMieghem_2002,Bhamidi_VanDerHofstad_Hooghiemstra_2011},
the inhomogeneous Bollobás--Riordan--Janson graph \cite{Kolossvary_Komjathy_2015}, and regular
random graphs and the configuration model~\cite{Bhamidi_VanDerHofstad_Hooghiemstra_2010,Bhamidi_VanDerHofstad_Hooghiemstra_2013+,Amini_Draief_Lelarge_2013}.
Discrete-time analogues of the aforementioned results are known for the complete graph \cite{Frieze_Grimmett_1985,Pittel_1987} and \Erdos--Rényi random graphs
\cite{Feige_Peleg_Raghavan_Upfal_1990,Fountoulakis_Huber_Panagiotou_2010}.
Salez \cite{Salez_2013} has recently derived an extension of \eqref{eq:PassageTimeHom} to random regular graphs.

\paragraph{Related work: Heterogeneous populations.}
Molchanov and Whitmeyer \cite{Molchanov_Whitmeyer_2010} derived an approximation of the broadcast time in the special case where every transmitter is initially informed of the selected message. In this case the model reduces to a version of the classical coupon collector's problem for which the limiting result is well known
\cite{Baum_Billingsley_1965}. Molchanov and Whitmeyer stated it as an open problem to extend the analysis to
more realistic scenarios where messages can be relayed. An early study on a related broadcasting model in a heterogeneous population was done by Daley and Kendall \cite{Daley_Kendall_1965} who studied the number of eventually informed nodes in a context where informed transmitters may stop transmitting upon contacting another informed node; see \cite{Lebensztayn_Machado_Rodriguez_2011} and references therein for generalizations and refinements.

\paragraph{Our contributions.}
In this paper we extend the results \eqref{eq:BroadcastTimeHom}--\eqref{eq:PassageTimeHom} to a context where (i) only a fraction $p$ of the nodes are transmitters for the selected message, and (ii) different individuals may have different contact rates. 
Under mild regularity assumptions, we will show that if the population size $n$ grows to infinity then the broadcast time distribution can be approximated by
\begin{equation}
 \label{eq:BroadcastTime}
 \Tsat
 \ \approx \
 \frac{1}{\lambda p} \left( \log(np) + \log n + V + \gumbelsm \right),
\end{equation}
and the joint distribution of the first passage times from the root to any fixed set of target nodes
$K$ by
\begin{equation}
 \label{eq:PassageTime}
 (\tau_i)_{i \in K} \ \approx \
 \frac{1}{\lambda p} \left( \log(np) + V + L_i \right)_{i \in K},
\end{equation}
where $\lambda$ is the mean contact rate of the transmitters and $V,\gumbelsm,L_i$ are independent random variables. When the contact rates of the transmitters are not identical, there seems to be no simple universal formula for the distribution of $V$ which captures the random fluctuations of the initial broadcasting phase. 
Instead, we provide a stochastic representation of $V$ which allows one to easily generate approximate samples from it. The approximations \eqref{eq:BroadcastTime}--\eqref{eq:PassageTime} are shown to be valid also when the fraction of transmitters $p = p_n$ depends on population size.

The rest of paper is organized as follows. In Section~\ref{sec:Model} we describe the model details and technical assumptions, and in Section~\ref{sec:Results} we state and discuss the main results. Detailed mathematical proofs of the results are given in Section~\ref{sec:Analysis}.

\section{Model description}
\label{sec:Model}

\subsection{Information dynamics}

Consider a population of nodes $\{0,1,\dots,n\}$ where node 0, called the \emph{root}, starts spreading copies of a new message to the other nodes. The nodes in the network are divided into \emph{transmitters} and \emph{receivers} of the selected message. We denote $\theta_i = 1$ if node $i$ is a transmitter, and $\theta_i = 0$ otherwise.
We assume that node $i$ makes contacts at the jump instants of a Poisson process of rate $Z_i$, independently of the others. Hence the population demographics is described by the numbers $(Z_i,\theta_i)$, $i=0,\dots,n$.

To describe the dynamics of the system, we denote $\Phi_i(t) = 1$ if node $i$ is \emph{informed} (has received a copy of the selected message) at time $t$, and $\Phi_i(t) = 0$ otherwise. When node $i$ makes a contact at time $t$, it selects a uniformly random target $j$ in the set of nodes excluding itself, independently of the past history of the system. If node $i$ is an informed transmitter ($\Phi_i(t-)=1$ and $\theta_i=1$), and $j$ is uninformed ($\Phi_j(t-) = 0$) when the contact is made, then a copy of the message gets transmitted from $i$ to $j$, and as a result of interaction $\Phi_j(t) = 1$ meaning node $j$ becomes informed instantly. Otherwise nothing happens at this event. 

We can make the above informal description precise by assuming that $(\Phi_0(t),\dots,\Phi_n(t))$
conditional on $(Z,\theta)$ is a right-continuous Markov process with values in $\{0,1\}^{n+1}$, 
initial state $\Phi(0) = (1,0,\dots,0)$, and transitions
\[
 \zeta \mapsto I_j \zeta
 \quad \text{at rate} \quad
 \frac{\sum_{i=0}^n \theta_i Z_i \zeta (1-\zeta)}{n},
 \quad j=0,\dots,n,n,
\]
where the map $I_j: \{0,1\}^{n+1} \to \{0,1\}^{n+1}$
updates the $j$-coordinate of its input to one and leaves the other coordinates unchanged. When the root node is a transmitter with a strictly positive contact rate $Z_0 > 0$, the system eventually gets absorbed into the state where all nodes are informed. The random time instant when this happens is called the \emph{broadcast time}, and denoted by
\begin{equation}
 \label{eq:DefBroadcast}
 \Tsat = \inf\{t \ge 0: \Phi(t) = (1,\dots,1) \}.
\end{equation}
In the terminology of first-passage percolation, this random quantity is often called the flooding time.
The \emph{first passage time} of node $i$ is denoted by
\begin{equation}
 \label{eq:DefPassage}
 \tau_i = \inf \left\{t \ge 0: \Phi_i(t) = 1 \right\}.
\end{equation}
This is the time instant when node $i$ becomes informed.

\subsection{Regular demographics}

In the sequel we will consider a sequence of models indexed by $n$, the number of initially
uninformed nodes. To emphasize the dependence on $n$ we often write
$\Phi^{(n)}(t)$ for the state of the $n$-th model, and so on. However, we might omit the superscript in places where
there is no risk of confusion.

To analyze the model as $n$ grows large, we need to make some regularity assumptions on the
demographics $(\theta_i^{(n)}, Z_i^{(n)})_{i=0}^n$ of the $n$-th model. To avoid the trivial case where nothing is ever transmitted, we assume that
\begin{equation}
 \label{eq:Demographics1}
 \theta_0^{(n)} = 1 \quad \text{and} \quad Z_0^{(n)} = z_0
 \quad \text{almost surely}
\end{equation}
for some $z_0>0$. We assume that the indicators $\theta_i^{(n)} \in \{0,1\}$ and the contact rates $Z_i^{(n)} \ge 0$ of nonroot nodes $i \ge 1$ are independent and such that $\pr(\theta_i^{(n)} = 1) = p_n$ and $Z_i^{(n)}$ is distributed according to a probability distribution $F$ on the nonnegative real numbers $\R_+$. To guarantee that there will be enough transmitters in a large population, we assume that the expected number of nonroot transmitters satisfies
\begin{equation}
 \label{eq:Demographics2}
 n p_n \gg n^\epsilon
\end{equation}
for some $\epsilon > 0$, where $a_n \gg b_n$ is shorthand for $a_n/b_n \to \infty$ as $n \to \infty$. To guarantee that the average contact rate of a large number of transmitters can be represented by a constant, we also assume that
\begin{equation}
 \label{eq:Demographics3}
 \lambda := \int z F(dz) > 0
 \quad \text{and} \quad
 \int z^2 F(dz) < \infty.
\end{equation}

To summarize, the $n$-th model is parameterized by a triplet $(z_0,p_n,F)$. Let us mention that the assumption that the contact rate of the root node $z_0$ is nonrandom is imposed for notational and conceptual convenience. Relaxing the assumption in the main results that follow is straightforward and left to the reader.

\section{Main results}
\label{sec:Results}

\subsection{Broadcast time}

Recall that the \emph{broadcast time} of the model is the time it takes for all nodes to become informed, defined by \eqref{eq:DefBroadcast}. To give a more detailed picture of information propagation, we decompose the broadcast time according to
\[
 \Tsat^{(n)} \, = \, \Thalf^{(n)} + \ThalfL^{(n)},
\]
where the \emph{first half-broadcast time}
\[
 \Thalf^{(n)} = \inf \left\{t \ge 0: \sum_{i=0}^n \Phi^{(n)}_i(t) \ge \frac{n+1}{2} \right\}
\]
is the time until half of the population becomes informed, and the second half-broadcast
time $\ThalfL^{(n)} = \Tsat\n - \Thalf\n$ is the time to inform the remaining half.

The following result gives a precise approximation of the broadcast time as $n$ tends to infinity. The symbols $\pto$ and $\dto$ stand for convergence in probability and convergence in distribution, respectively. We write $a_n \ll b_n$ if $a_n/b_n \to 0$ as $n \to \infty$. The symbol $\log x $ refers to the natural logarithm of $x$ to the base $e$.

\begin{theorem}
\label{the:BroadcastTime}
Assume that the population demographics parameterized by $(z_0,p_n,F)$ satisfies
\eqref{eq:Demographics1}--\eqref{eq:Demographics3}. Then
\[
 \left(\lambda p_n \Thalf\n - \log (p_n n), \ \lambda p_n \ThalfL\n - \log n \right)
 \ \dto \ (\alkusm,\gumbelsm),
\]
and in particular
\begin{equation}
 \label{eq:Main}
 \lambda p_n \Tsat\n - \log(p_n n) - \log n
 \ \dto \ \alkusm + \gumbelsm,
\end{equation}
where the random limits $\alkusm$ and $\gumbelsm$ are independent.
\end{theorem}

The limit $\gumbelsm$ in Theorem~\ref{the:BroadcastTime} captures the random fluctuations in the final phase of the broadcasting process and follows the standard Gumbel distribution
\[
 \pr(\gumbelsm \le t) = \exp(-e^{-t})
\]
commonly encountered in extreme value statistics.

The limit $V$ in Theorem~\ref{the:BroadcastTime} captures the random fluctuations in the initial phase. Unlike the final phase where the average contact rate of the transmitters is well concentrated around $\lambda$, the initial phase of  the broadcasting process is highly sensitive to the contact rates of the first transmitters. This is why there seems to be no simple universal analytical formula for the distribution of $V$. However, it turns out that $V$ can be represented as
\begin{equation}
 \label{eq:KickOff}
 V
 = \gamma + \sum_{k=1}^{\infty} \left(\frac{\uexp_k}{J_k} - \frac{1}{k} \right ),
\end{equation}
where $\gamma \approx 0.577$ is the Euler--Mascheroni constant, $\uexp_1,\uexp_2,\dots$ are independent with unit exponential distribution, and $J_k = \lambda^{-1} ( z_0 + Z_1 + \cdots + Z_{k-1})$ with $Z_1,Z_2,\dots$ being independent random numbers distributed according to $F$, also independent of the sequence $(\xi_1,\xi_2,\dots)$. In the special case where all contact rates $Z_i$ equal $\lambda$ with probability one, the distribution of $V$ reduces to a standard Gumbel.

To get a rough idea how much the fact that only a fraction $p$ of the nodes are transmitters slows down the information propagation, 
denote the leading nonrandom term in \eqref{eq:Main} by $B(n,p) = \frac{1}{\lambda p}( \log(pn) + \log n)$. Then
\[
 \frac{B(n,p)}{B(n,1)} = \frac{1}{p} \left( 1 - \frac{\log(1/p)}{2 \log n} \right),
\]
so for example, in a population of size $n=10^6$ where $p=1 \%$ are transmitters, we estimate that the broadcast time would be roughly 83 times higher compared to the population where all individuals were transmitters.

\subsection{First passage times}

Recall that the first passage time $\tau_i = \tau_i^{(n)}$ of node $i$ is the time instant at which node $i$ becomes informed, and is defined by \eqref{eq:DefPassage}.

\begin{theorem}
\label{the:FirstPassageTimes}
Assume that the population demographics parameterized by $(z_0,p_n,F)$ satisfies
\eqref{eq:Demographics1}--\eqref{eq:Demographics3}. Then for any fixed subset $K$ of nonroot nodes,
\[
 (\lambda p_n \tau_i\n - \log(p_n n))_{i \in K}
 \dto (\alkusm + L_i)_{i \in K},
\]
where $\alkusm$ has the representation~\eqref{eq:KickOff}, $L_i$ have the standard logistic
distribution
\[
 \pr(L_i \le t) = \frac{e^t}{e^t +1},
\]
and the random variables on the right side of the limit formula are independent.
\end{theorem}

\subsection{Proportion of informed nodes}

The proof of Theorem~\ref{the:FirstPassageTimes} utilizes the following auxiliary result which may
be interesting in its own right.

\begin{theorem}
\label{the:Logistic}
Assume that the population demographics parameterized by $(z_0,p_n,F)$ satisfies
\eqref{eq:Demographics1}--\eqref{eq:Demographics3}.
Then the proportion of informed nonroot nodes
\[
 \prop_n(t) = \frac{1}{n} \sum_{k=1}^n \Phi_k\n(t)
\]
can be approximated by the logistic distribution function 
\[
 \logis(t) = \frac{e^t}{e^t+1}
\]
according to
\[
 \sup_{t \in \RR} \left| \prop_n(\Thalf\n + t) - \logis(\lambda p_n t) \right| \ \pto \ 0.
\]
\end{theorem}

\section{Analysis}
\label{sec:Analysis}

We omit the floor notation in this section if no confusion is possible. This means that when a number like $n^\beta$ appears as an index or in a place where an integer is required it is implicitly rounded down to $\floor{n^\beta}$. We say that an event $A_n$ depending on $n$ occurs \emph{with high probability} if $\pr(A_n) \to 1$ as $n \to \infty$. Also, we write $X \eqst Y$ when $X$ and $Y$ have the same distribution, and $X_n = Y_n + o_P(1)$ when $X_n - Y_n \pto 0$ as $n \to \infty$.

\subsection{Representation of transmission times}

The time of the $k$-th successful message transmission can be written as
\[
 T_k\n = \inf\left\{ t \ge 0: \sum_{i=0}^n \Phi_i^{(n)}(t) = k+1 \right\},
\]
with the convention that $T^{(n)}_0 = 0$. From now on we will assume that the nodes are labeled in the order in which they become informed. This assumption has no effect on the joint distribution of the transmission times because
$(Z_k,\theta^{(n)}_k)_{k=1}^n$ is an exchangeable sequence and nodes choose targets
independently among all nodes.  At time $t \in [ T_k\n,  T_{k+1}\n)$ the total spreading
intensity equals
\begin{equation}
 \label{eq:SpreadingIntensity}
 R^{(n)}_k
 = z_0 + \sum_{j=1}^k \theta^{(n)}_j Z_j.
\end{equation}
Also, there are $n-k$ nodes yet to be informed, so the probability for the next connection to hit an uninformed node is $(n-k)/n$. Hence, conditional on  $(R_k\n)_{k=0}^{n}$, the interevent times $ T_{k+1}\n -  T_k\n$ are independent and exponentially distributed with rates $R_k\n \frac{n-k}{n}$. This reasoning allows us to represent the message transmission times as
\begin{equation}
 \label{eq:Tk}
 T^{(n)}_k \ = \ \sum_{j=0}^{k-1} \frac{n}{n-j} \frac{\xi_j}{R^{(n)}_j},
 \quad k=0,\dots,n,
\end{equation}
where $\xi_0,\xi_1,\dots$ are independent unit exponential random numbers, also independent of $\theta^{(n)}_i$ and $Z_i$, $i \ge 1$. With this representation, the broadcast time and the first half broadcast time can be expressed as
\[
 \Tsat^{(n)} = T^{(n)}_n
 \quad \text{and} \quad
 \Thalf^{(n)} = T^{(n)}_{\floor{n/2}}.
\]
\begin{remark*}
Strictly speaking, the equality in \eqref{eq:Tk} should be written as $\eqst$. Because our analysis is restricted to the distributions of the broadcast and first-passage times, we are free to choose the most convenient distributional representation for our needs, and in the future analysis we shall treat the variables $T^{(n)}_k$ as defined by \eqref{eq:Tk}.
\end{remark*}

\subsection{Concentration of total spreading intensity}

The following result confirms that the total spreading intensity $R^{(n)}_k$ defined by \eqref{eq:SpreadingIntensity} is well concentrated around its mean after sufficiently many message transmissions.

\begin{lemma}
\label{lem:An}
Assume that $p_n \gg n^{-2\alpha}$ for some $\alpha \in (0,\frac{1}{2})$. Let $c > 0$ and $\beta$ be
such that $\frac{1}{2} + \alpha < \beta < 1$. Then
\begin{equation}
 \label{eq:Claim}
 \left| \frac{R_k\n}{\lambda p_n k} - 1 \right| \le \frac{1}{\sqrt{p_n}n^\alpha }
 \quad \text{for all } c n^{\beta} \le k \le n
\end{equation}
with high probability as $n \to \infty$.
\end{lemma}
\begin{proof}
Denote the event in \eqref{eq:Claim} by $A_n$. Then
\begin{equation}
 \label{eq:Ancomp}
 \begin{aligned}
 \PP(A_n^c)
 &\le \PP\left(  |R_k\n - \lambda p_n k | > \lambda c \sqrt{p_n}n^{\beta-\alpha} \text{ for some }  c n^{\beta} \le k \le n\right) \\
 &\le \PP\left( \max_{1\le k \le n}  |R_k\n - \lambda p_n k| > \lambda c \sqrt{p_n}n^{\beta-\alpha} \right) \\
 &\le \PP\left(\max_{1\le k \le n} \Big | \sum_{i=1}^k \left(\uint_i\theta_i\n - \lambda p_n\right) \Big | > \lambda c \sqrt{p_n}n^{\beta-\alpha} - Z_0 \right).
 \end{aligned}
\end{equation}
Because
\[
 \Var \left(\sum_{i=1}^n \uint_i\theta_i\n \right)
 = n \left( \E (Z_1 \theta_1^{(n)})^2 - (\lambda p_n)^2 \right)
  \le  p_n n \E Z_1^2,
\]
we see by applying Kolmogorov's maximal inequality to the last term of \eqref{eq:Ancomp} that
\[
 \pr(A_n^c)
 \le
 \frac{p_n n \E Z_1^2}{\left( \lambda c \sqrt{p_n}n^{\beta-\alpha} - Z_0 \right)^2}.
\]
This shows that $\pr(A_n^c) \to 0$ because $p_n n \ll p_n n^{2(\beta-\alpha)}$ and we assumed that $\E Z_1^2 < \infty$.
\end{proof}

For normalized information propagation times 
we will use the notation $\Rsum(l,m) = \lambda p_n (T_m\n - T_l\n)$. 
Notice that it follows from \eqref{eq:Tk} that 
\begin{equation}
 \label{eq:Rsum}
 \Rsum(l,m) = \lambda p_n \sum_{k=l}^{m-1} \frac{n}{n-k} \frac{\xi_k}{R_k\n}.
\end{equation}
Now let
\[
 \Rsumm(l,m) = \sum_{k=l}^{m-1} \frac{n}{n-k} \frac{\xi_k}{k},
\]
which corresponds to replacing $R_k^{(n)}$ by $\lambda p_n k$ in~\eqref{eq:Rsum}. This leads us to the next lemma which is a natural corollary to Lemma \ref{lem:An}.

\begin{lemma}
\label{lem:An2}
Assume that $p_n \gg n^{-2\alpha}$ for some $\alpha \in (0,\frac{1}{2})$, and let $\beta$ be
such that $\frac{1}{2} + \alpha < \beta < 1$. Then there exists a number $\delta > 0$ such that 
\[
 \left| \frac{\Rsum(l,m)}{\Rsumm(l,m)} - 1 \right|
 \ \le \ n^{-\delta}
 \quad \text{for all } n^{\beta} \le l < m \le n
\]
with high probability. In particular,
\[
 n^{\delta/2} \left( \frac{\Rsum(a_n,b_n)}{\Rsumm(a_n,b_n)} - 1 \right) \ \pto \ 0
\]
for any integer sequences $(a_n)$ and $(b_n)$ such that $n^\beta \le a_n < b_n \le n$.
\end{lemma}
\begin{proof}
Choose $\alpha' > \alpha$ such that  $\alpha' + \frac{1}{2} < \beta$, and let $\delta = \frac{\alpha'-\alpha}{2}$. Then $c_n := \frac{1}{\sqrt{p_n} n^{\alpha'}} \ll n^{-\delta} \ll 1$. Let $A_n$ be the event that the ratio $r_{k,n} = \frac{R_k\n}{\lambda p_n k}$ satisfies
\[
 \left | r_{k,n} - 1 \right| \le c_n
\]
for all $n^\beta \le k \le n$. Because $c_n \le 1/2$ and $2 c_n \le n^{-\delta}$ for large $n$, it follows that on the event $A_n$, the corresponding inverse ratio is bounded by
\[
 \left| r_{k,n}^{-1} - 1 \right|
 = \frac{|r_{k,n} - 1|}{r_{k,n}}
 \le \frac{c_n}{1- c_n} \le 2 c_n \le n^{-\delta}.
\]
This implies the claim because $\pr(A_n) \to 1$ by Lemma~\ref{lem:An}.
\end{proof}

\subsection{Yule process in a random environment}
\label{subsec:Yule}

In the initial phase in a large population the first contacts are very likely to hit uninformed targets. If we assume
that all nodes are transmitters, we expect that the number of informed nodes in the initial phase will closely resemble the
following Yule process in a random environment. We will later show that the random number $V$ describing the long-run random factor in the Yule process will also characterize the randomness of the general spreading process in the initial phase until $n^\beta$ nodes are informed, with some $\beta \in (1/2,1)$.

Let $z_0>0$ be nonrandom and $Z_1,Z_2,\dots$ independent
nonnegative random numbers with a common distribution $F$. We consider a birth process $N(t)$ on the positive integers
such that $N(0)=1$ and conditional on $(Z_1,Z_2,\dots)$ the process $N$ is Markov where transition $k \mapsto k+1$ occurs at rate $z_0+Z_1+\cdots Z_{k-1}$.
This models the size of a population where all individuals live forever, initially there is one individual producing children at rate $z_0$,
and the $k$-th born individual produces children at rate $Z_k$. Let $T_k$ be the time of the $k$-th birth, and denote $\xi_k = \lambda(T_{k}-T_{k-1})J_k$,
where
\begin{equation}
 \label{eq:JDef}
 J_k = \lambda^{-1} \left( z_0 + \sum_{j=1}^{k-1} Z_j \right), \quad k \ge 1.
\end{equation}
Then the time at which the $m$-th birth occurs can be written as
\[
 T_m = \lambda^{-1} \sum_{k=1}^{m} \frac{\xi_k}{J_k},
\]
and the Markov assumptions imply that $\xi_1,\xi_2,\dots$ are unit exponential random numbers independent of each other and the sequence  $(Z_1,Z_2,\dots)$.

\begin{proposition}
\label{the:Yule}
Assume that $\lambda = \int z F(dz) > 0$ and $\int z \, |\log z| \, F(dz) < \infty$. Then $\lambda T_m - \log m \to V$ almost surely, where the limit  can be represented as
\[
  V = \gamma + \sum_{k=1}^{\infty} \left( \frac{\uexp_k}{J_k} - \frac{1}{k} \right),
\]
where $\gamma \approx 0.577$  is the Euler--Mascheroni constant.
\end{proposition}

Note that when $z_0=1$ and $F$ has all its mass at one, $V$ has standard Gumbel distribution. Note
also that $t |\log t| \le 1 \vee t^2$ for all $t \ge 0$, which shows that $\E Z |\log Z| < \infty$ whenever $Z$ has a finite second moment.

\begin{proof}
Write
\begin{equation}
 \label{eq:Yule}
  \lambda T_m
  = \sum_{k=1}^{m} \frac{\uexp_k}{J_k}
  \ = \ \sum_{k=1}^{m} \frac{\uexp_k-1}{J_k}
    + \sum_{k=1}^{m} \left( \frac{1}{J_k}-\frac{1}{k} \right ) + \sum_{k=1}^{m} \frac{1}{k}.
\end{equation}
Denote by $A_m$ the first sum on the right of \eqref{eq:Yule}. By noting that $0 < J_k^{-1} \le \lambda/z_0$ for all $k \ge 1$, we see that the process $m \to A_m$ is a martingale with mean zero and variance
\[
 \E A_m^2
 = \sum_{k=1}^{m} \E \left( \frac{\xi_k-1}{J_k} \right)^2
 = \sum_{k=1}^{m} \E \left( \frac{1}{J_k} \right)^2
 \le c \sum_{k=1}^{m} \frac{1}{k^2},
\]
where the constant $c = \sup_{k \ge 1} \E \left( \frac{k}{J_k} \right)^2$ is finite by \cite[Lemma 3]{Athreya_Karlin_1967}.
We conclude that $A_m$ is a martingale bounded in square mean and thus converges almost surely.

By \cite[Thm.~1]{Athreya_Karlin_1967}, the second sum on the right of \eqref{eq:Yule} converges almost surely. The claim thus follows from the fact that $\sum_{k=1}^{m} \frac{1}{k} - \log m \to \gamma$.
\end{proof}

\subsubsection{Thinned Yule process}
\label{sec:ThinnedYule}
For future needs it will be helpful to analyze a modification of the above Yule process where only a fraction of all individuals are able to reproduce. Fix some nonrandom constants $p \in [0,1]$ and $z_0>0$, and let $Z_1,Z_2,\dots$ independent
nonnegative random numbers with a common distribution $F$. Also, let $\theta_1,\theta_2,\dots$ be independent $\Ber(p)$-distributed random integers, and let $\xi_0,\xi_1,\dots$ be independent unit exponential random numbers. We also assume that the sequences $(Z_i), (\theta_i), (\xi_i)$ are independent. Define
\[
 R_k = z_0 + \sum_{j=1}^k \theta_j Z_j, \quad k \ge 0,
\]
and
\[
 T_m = \sum_{\ell=0}^{m-1} \frac{\xi_\ell}{R_\ell}, \quad m \ge 0,
\]
with the convention that an empty sum is zero. Then we can interpret $T_m$ as the time until the $m$-th birth takes place in a population where all individuals live forever, initially there is one individual that produces offspring at rate $z_0$, the $k$-th born individual tries to produce offspring at rate $Z_k$, and the $k$-th born individual is able to produce offspring if and only if $\theta_k = 1$. Note that $\theta_k=1$ for all $k$ almost surely in the special case where $p=1$, and then this model reduces back to the previously defined Yule process.

If we only care about the individuals that are able to produce offspring, we may define
\[
 \hat T_m = T_{D_m}, \quad m \ge 0,
\]
where
\[
 D_m = \inf\{k \ge 0: \sum_{j=1}^k \theta_j = m\}, \quad m \ge 0.
\]
When the individuals are labeled in the order that they are born, $D_m$ indicates the label of the $m$-th born individual among those able to reproduce. The following result shows how the birth times of the reproduction-capable individuals can be analyzed using the same formula as in Proposition~\ref{the:Yule}.

\begin{lemma}
\label{the:Yule2}
For any integer $m \ge 0$,
\[
 \hat T_m \eqst (\lambda p)^{-1} \sum_{\ell=1}^{m} \frac{\xi_\ell}{J_\ell},
\]
where $(J_1,J_2,\dots)$ is defined by \eqref{eq:JDef} and independent of $(\xi_1,\xi_2,\dots)$.
\end{lemma}
\begin{proof}
Note that $R_k$ equals the net reproduction rate after $k$ births. Because this net rate only includes contributions from those individuals able to reproduce, we see that  $R_k = R_{D_\ell}$ for $D_\ell \le k < D_{\ell+1}$. Here $R_{D_\ell}$ equals the net reproduction rate after $\ell$ reproduction-capable individuals have been born. Because these individuals are indexed by $D_1,\dots,D_\ell$, we see that
\[
 R_{D_\ell} = z_0 + \sum_{k=1}^\ell Z_{D_k}.
\]
By denoting $\hat R_\ell = R_{D_\ell}$ and $\hat Z_\ell = Z_{D_\ell}$, this can be rephrased as
\[
 \hat R_{\ell} = z_0 + \sum_{k=1}^\ell \hat Z_k.
\]
By expressing $\hat T_m$ as 
\[
 \hat T_m
 = \sum_{\ell = 0}^{m-1} ( \hat T_{\ell+1} - \hat T_\ell )
 = \sum_{\ell = 0}^{m-1} \sum_{k=D_\ell}^{D_{\ell+1}-1} \frac{\xi_k}{R_k}
 = \sum_{\ell = 0}^{m-1} \frac{1}{R_{D_\ell}} \left( \sum_{k=D_\ell}^{D_{\ell+1}-1} \xi_k \right),
\]
we find that
\[
 \hat T_m
 = \sum_{\ell = 0}^{m-1} \frac{\hat \xi_\ell}{\hat R_\ell},
\]
where $\hat \xi_\ell = \sum_{k=D_\ell}^{D_{\ell+1}-1} \xi_k$. Because $\hat \xi_\ell$ equals a random sum of independent unit exponential random numbers, and the number of terms in the sum is independent of the summands and follows a geometric distribution on $\{1,2,\dots\}$ with success probability $p$, it follows that the distribution of $\hat \xi_\ell$ is exponential with mean $p^{-1}$. The independence of $(\xi_i)$ and $(\theta_i)$ also implies that $\hat \xi_0, \hat \xi_1, \dots$ are independent, and hence $(\hat \xi_0, \hat \xi_1, \dots) \eqst (p^{-1} \xi_1, p^{-1} \xi_2, \dots)$.

To finish the proof, it suffices to verify that $(\hat R_0, \hat R_1, \dots) \eqst (\lambda J_1, \lambda J_2, \dots)$,
and that the sequences $(\hat Z_\ell)$ and $(\hat \xi_\ell)$ are independent. Both of these facts follow from the assumption that $(Z_i)$ forms an independent and identically distributed sequence that is independent of $(\theta_i)$ and $(\xi_i)$.
\end{proof}

\subsection{Broadcast time}

Now we will prove Theorem~\ref{the:BroadcastTime}. By assumption \eqref{eq:Demographics2} we may choose an $\epsilon > 0$ such that $p_n n \gg n^{\epsilon}$. Then necessarily $\epsilon \in (0,1)$ because $p_n \le 1$.
Let $\alpha = (1-\epsilon)/2$ and fix a number $\beta$ such that  $\frac{1}{2}+\alpha < \beta < 1.$ Then, using \eqref{eq:Tk}, write
\[
\begin{aligned}
 &\left( \lambda p_n \Thalf\n, \ \lambda p_n \ThalfL\n \right) \\
 = & \left( \Rsum(0,n^\beta)+\Rsum(n^\beta,n/2),\  \Rsum(n/2,n-n^\beta)+\Rsum(n-n^\beta,n) \right).
\end{aligned}
\]
Lemma \ref{lem:alkuosa} below shows that 
\[
 \Rsum(0,n^\beta) - \log (p_n n^\beta) 
 \ \dto \ \alkusm.
\]
In Lemma \ref{lem:keskiosa} below we show that 
\begin{equation}
\label{eq:Rsumm1}
\Rsumm(n^\beta,n/2) = (1-\beta)\log n + o_p(1)
\end{equation}
and
\begin{equation}
\label{eq:Rsumm2}
\Rsumm(n/2,n-n^\beta) = (1-\beta)\log n + o_p(1).
\end{equation}
In Lemma \ref{lem:Final} it is shown that
\begin{equation}
\label{eq:Rsumm3}
\Rsumm(n-n^\beta,n) -  \beta \log n \ \dto \gumbelsm.
\end{equation}

Equations \eqref{eq:Rsumm1}--\eqref{eq:Rsumm3} are true even if we replace $\Rsumm$ with $\Rsum$. This can be seen, for example, by writing

\begin{equation}\label{eq:RReplace}
	\begin{aligned}
	&\Rsum(n^\beta,n/2) - (1-\beta) \log n \\
= &\Rsumm(n^\beta,n/2) - (1-\beta) \log n + X_nY_n,
	\end{aligned}
\end{equation}
where 
\[
X_n = n^\delta\left(\frac{\Rsum(n^\beta,n/2)}{\Rsumm(n^\beta,n/2)}-1\right)
\]
and
\[
Y_n = \frac{\Rsumm(n^\beta,n/2)}{n^\delta}.
\]
For a suitably small $\delta > 0$ Lemma \ref{lem:An2} implies that $X_n = o_p(1)$. Also, it is clear from equation \eqref{eq:Rsumm1} that $Y_n = o_p(1)$. 
Then we can get the desired result by applying Slutsky's lemma (e.g \cite{VanDerVaart_2000} Lemma 2.8) to the last line of \eqref{eq:RReplace}. Equations \eqref{eq:Rsumm2} and \eqref{eq:Rsumm3} are handled in the same manner.

Finally, notice that $\Rsumm(n-n^\beta,n)$ is independent of $ \Rsum(0,n^\beta)$ so the joint convergence to the independent pair $(V,\gumbelsm)$ in Theorem~\ref{the:BroadcastTime} is true as well.

\begin{lemma} [Initial phase]
\label{lem:alkuosa}

Assume that $p_n \gg n^{-2\alpha}$ for some $\alpha \in (0,\frac{1}{2})$, and let $\beta$ be
such that $\frac{1}{2} + \alpha < \beta < 1$.  Then
\[
 \Rsum(0,n^\beta) - \log (p_n n^\beta) 
 \ \dto \ \alkusm.
\]
\end{lemma}

\begin{proof}
For convenience, we will assume that $Z_1,Z_2,\dots$ and $\theta^{(n)}_1,\theta^{(n)}_2,\dots$ are \emph{infinite} independent sequences of i.i.d.\ random numbers distributed according to $F$ and $\Ber(p_n)$, respectively, although only the first $n$ elements are required to construct the model with $n$ nonroot nodes. From now on also the spreading intensity $R_k\n$ will be defined by \eqref{eq:SpreadingIntensity} for all integers $k \ge 0$. As in Section~\ref{subsec:Yule}, we define $J_1 = \lambda^{-1}z_0$ and
\[
 J_k = \lambda^{-1} \left( z_0 + \sum_{i=1}^{k-1} \uint_i \right), \quad k \ge 2.
\]
Also, we define $D_0^{(n)} = 0$ and 
\[
 D_m^{(n)} = \inf\left\{k \ge 1 : \sum_{i=1}^{k} \theta_i\n = m\right\}, \quad m \ge 1.
\]
Notice that $D_m^{(n)}$ counts how many nonroot nodes need to be informed in order to get $m$ nonroot transmitters informed.  Because exactly one node is informed on every successful message transmission, $D_m^{(n)}$ is also the number of successful message transmissions required to inform $m$ nonroot transmitters.

We prove the result by rigorously verifying each step in the chain of approximations
\begin{align*}
 \Rsum(0,n^\beta)
 \approx \lambda p_n \sum_{k=0}^{n^\beta-1} \frac{\xi_k}{R_k\n}
 \approx \lambda p_n \!\!\!\!\! \sum_{k=0}^{D^{(n)}_{r_n}-1} \!\!\!\! \frac{\xi_k}{R_k\n}
 \approx \sum_{k=1}^{r_n} \frac{\xi_k}{J_k}
 \approx \log (p_n n^\beta) + V,
\end{align*}
where $r_n = \floor{p_n n^\beta}$ approximates the mean number of transmitters among the first $n^\beta$ informed nodes.

(i) Assume that $n$ is large enough so that $n - n^\beta + 1 \ge \frac{1}{2}n$. Then
\[
 1
 \le \frac{n}{n-k}
 \le \frac{n}{n - n^\beta + 1}
 =   1 + \frac{n^\beta - 1}{n - n^\beta + 1}
 \le 1 + \frac{2}{n^{1-\beta}}
\]
for all $0 \le k \le n^\beta-1$. Therefore 
\begin{equation}
 \label{eq:Step1}
 \begin{aligned}
   &(\lambda p_n)^{-1}\Rsum(0,n^\beta) 
  = \sum_{k=0}^{n^\beta-1} \frac{n}{n-k} \frac{\uexp_k}{R_k\n} \\
 = & \left( 1 + O(n^{-(1-\beta)}) \right) \sum_{k=0}^{n^\beta-1}\frac{\uexp_k}{R_k\n},
 \end{aligned}
\end{equation}
where first equation is the definition of $\Rsum(0,n^\beta)$.

(ii) First write the estimate
\begin{equation}\label{eq:summanvaihto1}
\begin{aligned}
 &\left| \sum_{k=0}^{D_{r_n}^{(n)} - 1} \frac{\uexp_k}{R_k\n} - \sum_{k=0}^{n^{\beta}-1} \frac{\uexp_k}{R_k\n} \right| \\
 =& \sum_{k=n^\beta \wedge D_{r_n}^{(n)}}^{n^\beta \vee D_{r_n}^{(n)} - 1} \frac{\uexp_k}{R_k\n} 
 \le \frac{1}{R_{n^\beta \wedge D_{r_n}^{(n)}}\n}  \sum_{k=n^\beta \wedge D_{r_n}^{(n)}}^{n^\beta \vee D_{r_n}^{(n)} - 1} \uexp_k = 
\frac{|D_{r_n}^{(n)}-n^\beta|}{R_{n^\beta \wedge D_{r_n}^{(n)}}\n} O_P(1).
\end{aligned}
\end{equation}
Now choose an $\epsilon > 0$ such that $\beta - \epsilon > \frac{1}{2} + \alpha$. Then
$p_n n^{\beta - \epsilon} \to \infty$, so Lemma~\ref{lem:taupnn} below shows that with high probability
\begin{equation*}
\frac{|D_{r_n}^{(n)} - n^\beta|}{R_{n^\beta \wedge D_{r_n}^{(n)}}\n}
\le \frac{n^{\beta-\epsilon/2}}{R_{n^\beta/2}\n},
\end{equation*}
and Lemma~\ref{lem:An} further implies that $R_{n^\beta/2}\n \ge \frac{1}{3} \lambda p_n n^{\beta}$
with high probability. By combining these estimates with \eqref{eq:summanvaihto1} we see that
\begin{equation}
 \label{eq:Step2}
 \sum_{k=0}^{D_{r_n}^{(n)} - 1} \frac{\uexp_k}{R_k\n}
 = \sum_{k=0}^{n^{\beta}-1} \frac{\uexp_k}{R_k\n} + o_P(1/p_n).
\end{equation} 

(iii) Because the definitions of $D_m\n$ and $R_k\n$ coincide with the thinned Yule processes analyzed in Section~\ref{sec:ThinnedYule}, Lemma~\ref{the:Yule2} implies that for all $m \ge 1$,
\begin{equation}
\label{eq:Step3}
\lambda p_n \sum_{k=0}^{D_m\n-1} \frac{\uexp_k}{R_k\n} \eqst \sum_{k=1}^{m} \frac{\uexp_k}{J_{k}}.
\end{equation} 

(iv) The analysis of the Yule process in Proposition~\ref{the:Yule} shows that with probability one, as $m \to \infty$,
\begin{equation}
 \label{eq:Step4}
  \sum_{k=1}^{m} \frac{\uexp_k}{J_k} \ - \ \log m
  \ \to \ V.
\end{equation}

We can now finish the proof by combining the above steps. Formulas
\eqref{eq:Step3}--\eqref{eq:Step4} imply that
\[
 \lambda p_n \sum_{k=0}^{D_{r_n}-1} \frac{\uexp_k}{R_k\n} - \log (p_n n^\beta)
 \ \dto \ V,
\]
and by \eqref{eq:Step2} the same limit also holds if we replace $D_{r_n}$ by $n^\beta$ above.
The claim now follows by \eqref{eq:Step1} because $0 \le n^{-(1-\beta)} \log (p_n n^\beta) \to 0$.
\end{proof}

\begin{lemma}
\label{lem:taupnn}
Let $A_n = \sum_{k=1}^{r_n} X_k\n$, where $r_n = \floor{p_n n^\beta}$ with $\beta \in (0,1)$, and the summands are
independent geometrically distributed random numbers on $\{1,2,\dots\}$ with success probability
$p_n$. Assume that $p_n n^\beta \gg n^\epsilon$ for some $\epsilon > 0$. Then
\[
 | A_n - n^\beta | \le n^{\beta-\epsilon/2}
 \qquad \text{and} \qquad
 n^\beta \wedge A_n \ge \frac{n^\beta}{2}
\]
with high probability.
\end{lemma}
\begin{proof}
Notice that $\E A_n = r_n/p_n$ and $\Var(A_n) = r_n \frac{1-p_n}{p_n^2} \le n^\beta/p_n$. Note that $|n^\beta - r_n/p_n| \le 1/p_n$, so that
\[
 |A_n - n^\beta|
 \ \le \ |A_n - \E A_n| + 1/p_n.
\]
Assume that $n$ is so large that $p_n n^\beta \ge 2 n^{\epsilon/2}$. Then $n^{\beta-\epsilon/2} -
1/p_n \ge \frac{1}{2} n^{\beta - \epsilon/2}$, and Chebyshev's inequality shows that
\[
 \pr \left(  | A_n - n^\beta | > n^{\beta-\epsilon/2} \right)
 \le \pr \left( |A_n - \E A_n | > n^{\beta-\epsilon/2} - \frac{1}{p_n} \right )
 \le \frac{4 n^\beta/p_n}{n^{2\beta - \epsilon}}.
\]
This proves the first claim because the right side above vanishes as $n$ grows to infinity. The second claim follows from the first by choosing $n$ large enough so that $1-n^{-\epsilon/2} \ge \frac{1}{2}$.
\end{proof}

\begin{lemma}[Middle phase]
\label{lem:keskiosa}
Let $\beta \in (0,1).$ Then,
\begin{align*}
 \Rsumm(n^\beta,n/2)
 \ &= \ (1-\beta) \log n + o_P(1) \\
 \text{and} \\
 \Rsumm(n/2,n-n^\beta)
 \ &= \ (1-\beta) \log n + o_P(1).
\end{align*}
\end{lemma}
\begin{proof}
First we calculate that
\[
 \Var \Rsumm(n^\beta,n/2) 
 \le \sum_{k=n^\beta}^{\frac{n}{2}-1} \frac{1}{k^2} \left(\frac{n}{n-k}\right)^2
 \le 4 \sum_{k=n^\beta}^\infty \frac{1}{k^2}
 = o(1),
\]
and 
\begin{align*}
 &\E \Rsumm(n^\beta,n/2) =  \sum_{k=n^\beta}^{\frac{n}{2}-1} \left(\frac{1}{k} + \frac{1}{n-k} \right)
 \\ &= \sum_{k=n^\beta}^{\frac{n}{2}-1} \frac{1}{k} + \sum_{k=\frac{n}{2}+1}^{n-n^\beta}  \frac{1}{k} =  \ (1-\beta)\log n + o(1).
\end{align*}
To justify the last equality above note that if $b_n > a_n \to \infty$ as $n \to \infty$ then $\sum_{k=a_n}^{b_n} \frac{1}{k} = \log b_n - \log a_n + o(1)$. 

Then first claim in the lemma can now be then obtained by Chebyshev's inequality. The second claim is proved similarly.
\end{proof}

\begin{lemma}[Final phase]
\label{lem:Final}
Let $\beta \in (0,1)$. Then
\[
 \Rsumm(n-n^\beta,n) - \log(n^\beta)
 \ \dto \ \gumbelsm,
\]
where the limit has the standard Gumbel distribution $\pr(\gumbelsm \le t) = e^{-e^{-t}}$.
\end{lemma}
\begin{proof}
Denote $b_n = \floor{n^\beta}$. Observe first that
\[
 \Rsumm(n-b_n,n)
 \ = \sum_{k= n  - b_n}^{n-1} \frac{n}{n-k} \frac{\uexp_k}{k}
 \ = \ \sum_{i=1}^{b_n} \frac{n}{n-i} \frac{\uexp_{n-i}}{i}
 = M_n \sum_{i=1}^{b_n}\frac{\uexp_{n-i}}{i},
\]
where $1 \le M_n \le \frac{n}{n-b_n}$ almost surely. Next, Rényi's representation formula of exponential order statistics (e.g.\ \cite[Thm~4.6.1]{Arnold_Balakrishan_Nagaraja_2008}) implies that $\sum_{i=1}^m \frac{\xi_i}{i} \eqst \max\{\xi_1,\dots,\xi_m\}$, so that
\[
 \pr\left( \sum_{i=1}^m \frac{\xi_i}{i} - \log m \le t \right)
 = \pr( \xi_1 \le t + \log m)^m
 = \left( 1 - \frac{e^{-t}}{m} \right)^m
 \to e^{-e^{-t}},
\]
as $m \to \infty$. This implies the claim because $\sum_{i=1}^{b_n}\frac{\uexp_{n-i}}{i} \eqst \sum_{i=1}^{b_n} \frac{\xi_i}{i}$, and because $M_n \to 1$ and $(M_n-1)\log b_n \to 0$ almost surely.
\end{proof}

\subsection{Fraction of informed nodes}

In this section we prove Theorem~\ref{the:Logistic} by first proving two auxiliary lemmas. Below
we use the convention that $S_n(t) = 0$ for $t < 0$.

\begin{lemma}
\label{lem:F}
For any $t \in \R$,
\[
 \lambda p_n\left(T_{\logis(t)n}\n - \Thalf\n\right) \pto t.
\]
\end{lemma}
\begin{proof}
For $t=0$ the claim is obvious because $S(0) = \frac{1}{2}.$
Assume then that $t < 0$, so that $0 < S(t) < 1/2$. Using the representation \eqref{eq:Tk} we can write
\[
 -\lambda p_n\left(T_{\logis(t)n}\n - \Thalf\n\right) \ = \ \Rsum\left(\logis(t)n,n/2\right).
\]
Lemma \ref{lem:An2} implies that
\[
 \frac{\Rsum\left(\logis(t)n,n/2\right)}{\Rsumm\left(\logis(t)n,n/2\right)} \pto 1. 
\]
Therefore, by Slutsky's lemma, it suffices to show that
\[
 \Rsumm\left(\logis(t)n,n/2\right) \pto -t.
\]
This is true because
\[
 \begin{aligned}
 & \EE \left [ \Rsumm\left(\logis(t)n,n/2\right) \right ] \\
 =& \sum_{k=\logis(t)n}^{\frac{n}{2}-1} \frac{n }{k(n-k)}
 =  \sum_{k=\logis(t)n}^{\frac{n}{2}-1} \left( \frac{1}{k} + \frac{1}{n-k} \right)
 =  \sum_{k=\logis(t)n}^{(1-\logis(t))n} \frac{1}{k} + o(1)\\
 =& \log (1-\logis(t)) - \log(\logis(t))+o(1)
 = -t + o(1)
 \end{aligned}
\]
and
\[
 \Var \left [ \Rsumm\left(\logis(t)n,n/2\right) \right ]
 = \sum_{k=\logis(t)n}^{\frac{n}{2}-1} \frac{1}{k^2}\left(\frac{n}{n-k}\right)^2
 \le 4 \sum_{k=\logis(t)n}^\infty \frac{1}{k^2}
 = o(1).
\]
For $t > 0$ the proof is similar and hence omitted. 
\end{proof}

\begin{lemma}
\label{lem:H}
For any $t \in \RR$,
\[
 \prop_n ( \Thalf\n + t/(\lambda p_n) ) \pto \logis(t).
\]
\end{lemma}
\begin{proof}
Note first that for any $t \in \R$ and $r \in (0,1)$,
\[
 \prop_n(\Thalf\n+t) \ge r
 \ \Leftrightarrow \
 n \prop_n(\Thalf\n+t) \ge \lfloor n r \rfloor
 \ \Leftrightarrow \
 T_{\floor{nr}}\n - \Thalf\n \le t.
\]
Now fix $t \in \R$ and assume that $\epsilon > 0$ is so small that $S(t)-\epsilon$ and $S(t) +
\epsilon$ are in $(0,1)$. Let $t_1 < t < t_2$ be such that $S(t_1) = S(t)-\epsilon$ and $S(t_2) =
S(t) + \epsilon$. Then the above equivalence shows that
\begin{equation*}
\begin{aligned}
 &\PP\left(|\prop_n(\Thalf\n+t/(\lambda p_n))-\logis(t)| > \epsilon \right) \\
 =& \PP\left( \prop_n(\Thalf\n+t/(\lambda p_n)) < S(t_1) \right)
 + \PP\left( \prop_n(\Thalf\n+t/(\lambda p_n)) \ge S(t_2) \right) \\
 \le & \PP\left(\lambda p_n(T_{\floor{S(t_1)n}}\n - \Thalf\n) > t \right)
 + \PP\left(\lambda p_n(T_{\floor{S(t_2)n}}\n - \Thalf\n) \le t \right).
\end{aligned}
\end{equation*}
By Lemma \ref{lem:F}, $\lambda p_n(T_{\floor{S(t_i)n}}\n - \Thalf\n) \pto t_i$ for $i=1,2$. This
implies that both terms on the last line above vanish as $n$ grows.
\end{proof}

\begin{proof}[Proof of Theorem \ref{the:Logistic}]
Let $\epsilon > 0$, and fix a large enough $M>0$ such that $S(M) \ge 1-\epsilon/2$ and $S(-M) \le \epsilon/2$.
Then partition $[-M,M]$ into $N$ subintervals using equidistant points $-M=t_0 < t_1 < t_2 < \cdots < t_N = M$, where $N$ is so large that the subinterval length is bounded by $2 M/N < \epsilon/2$. Because $S$ is increasing and Lipschitz with $0 \le S'(t) \le 1$ for all $t$, it is not hard to verify that for any other increasing function $f: \R \to [0,1]$,
\[
 \sup_{t \in \R} | S(t) - f(t) | \ \le \ \epsilon/2 + \max_{i=0,\dots,N} |S(t_i) - f(t_i)|.
\]
By applying the above estimate to the random function $t \mapsto S_n(\Thalf^{(n)}+t/(\lambda p_n))$, we see with the help of Lemma~\ref{lem:H} that
\[
\begin{aligned}
 &\PP\left(\sup_{t\in \RR}|S_n(\Thalf\n+t/(\lambda p_n))-S(t)| > \epsilon\right) \\
 \le &\sum_{i=0}^N \PP\left(|S_n(\Thalf\n+t_i/(\lambda  p_n))-S(t_i)| > \frac{\epsilon}{2}\right) \to 0.
\end{aligned}
\]
\end{proof}

\subsection{First passage times}
\begin{proof}[Proof of Theorem~\ref{the:FirstPassageTimes}]
Denote $$\alkusm_n = \lambda p_n \Thalf\n - \log (p_n n).$$

For notational purposes, let us assume that $K=\{1,\dots,k\}.$
Now write
$$\lambda p_n\tau_i - \log (p_n n) = \alkusm_n + \lambda p_n\left(\tau_i-  \Thalf\n\right).$$
In the following we will prove that
$$\left(\lambda p_n(\tau_1-  \Thalf\n),\dots,\lambda p_n(\tau_k-  \Thalf\n),\alkusm_n\right) \dto (L_1,\dots,L_k,\alkusm)$$ which implies the claim of Theorem~\ref{the:FirstPassageTimes}.

Pick $t \in \RR$ and $a=(a_1,\dots,a_k) \in \RR^k.$ By relabelling the elements of $K$ we can
assume that $a_1 \le a_2 \le \cdots \le a_k$. Define the event
\begin{align*}
 A_n
 :=& \bigcap_{i=1}^k \{\lambda p_n(\tau_i - \Thalf\n) \le a_i \}\\
  =& \bigcap_{i=1}^k \{\text{Node } i \text{ is informed at time } \Thalf\n + a_i/(\lambda p_n)\}.
\end{align*}

Now we need to prove that
\[
 \PP(A_n,\alkusm_n \le t) \to \logis(a_1)\cdots \logis(a_k) \pr(V \le t),
\]
where $V$ is the random variable introduced in Proposition \ref{the:Yule}. Let $\mathcal G\n =
\sigma(\prop_n(t),t\ge0)$. Because the order in which the nodes are informed does not depend on
their labeling it holds that
\[
 \EE\left[\ind_{A_n} | \mathcal G\n \right]
 = 0 \vee \prod_{i=1}^k \left(\prop_n(\Thalf\n+a_i/(\lambda p_n))-n^{-1}(i-1) \right )  =: J_a\n.
\]
Using the above formula and noticing that $\alkusm_n$ is $\mathcal G\n$-measurable, we see that
\begin{equation}
\label{eq:Ga}
 \PP(A_n, \alkusm_n \le t)
 = \EE \left[ \ind_{\{\alkusm_n \le t\}} \EE\left( \ind_{A_n} \, | \, \mathcal G\n \right ) \right]
 = \EE J_a\n\ind_{\{\alkusm_n \le t\}}.
\end{equation}
From the previous equation we find that
\begin{equation}\label{eq:ehEE}
\begin{aligned}
 &\left| \PP(A_n, \alkusm_n \le t) - \pr(V \le t) \prod_{i=1}^k \logis(a_i) \right| \\
 =&\left|\EE \left( J_a\n - \frac{\pr(V \le t)}{\PP(\alkusm_n \le t)} \prod_{i=1}^k \logis(a_i) \right) \ind_{\{\alkusm_n \le t\}} \right| \\
 \le &\EE \left| J_a\n - \frac{\pr(V \le t)}{\PP(\alkusm_n \le t)} \prod_{i=1}^k \logis(a_i) \right|.
\end{aligned}
\end{equation}

Finally, we get from Lemma \ref{lem:H} that $J_a\n \pto \logis(a_1)\cdots \logis(a_k)$ and from
Theorem~\ref{the:BroadcastTime} that $\PP(\alkusm_n \le t) \to \pr(V \le t)$. Hence
\[
 J_a\n - \frac{\pr(V \le t)}{\PP(\alkusm_n \le t)} \prod_{i=1}^k \logis(a_i)
 \ \pto \ 0.
\]
Because $\PP(\alkusm_n \le t) \ge \frac{1}{2} \pr(V \le t)$ for all large enough $n$, we see that
the above random numbers are almost surely bounded by 2 in absolute value for all large enough
$n$. Therefore the last line in~\eqref{eq:ehEE} vanishes as $n \to \infty$, and
Theorem~\ref{the:FirstPassageTimes} is proved.
\end{proof}

\subsection*{Acknowledegements}
The main body of research reported in this article was carried out when both authors were working at the University of Jyväskylä, Finland. Financial support from Emil Aaltonen Foundation is gratefully acknowledged. We thank two anonymous referees whose comments have helped to improve the presentation of the article.

\bibliography{lslReferences}
\bibliographystyle{abbrv}
\end{document}